\documentclass{article}

\usepackage{authblk}
\usepackage{amsmath}
\usepackage{amsthm}
\usepackage{amssymb}
\usepackage{color}
\usepackage{comment}
\usepackage{graphicx,float}
\usepackage{graphics}
\usepackage{enumitem}
\usepackage{mathtools}
\usepackage[mathlines]{lineno}
\usepackage[normalem]{ulem}

\newtheorem{theorem}{Theorem}[section]
\newtheorem{lema}[theorem]{Lemma}
\newtheorem{corollary}[theorem]{Corollary}

\newtheorem{definition}[theorem]{Definition}

\title{On some classes of bivalent and trivalent planar graphs}

\author[1]{\normalsize{Jorge Alencar}\thanks{jorgealencar@iftm.edu.br}}
\author[2]{\normalsize{Jean-Guy Caputo }\thanks{jean-guy.caputo@insa-rouen.fr}}
\author[3]{\normalsize{Leonardo de Lima} \thanks{leonardo.delima@ufpr.br}}
\author[2]{\normalsize{Arnaud Knippel} \thanks{arnaud.knippel@insa-rouen.fr}}

\affil[1]{Instituto Federal de Educa\c{c}\~{a}o, Ci\^{e}ncia e Tecnologia do Tri\^angulo Mineiro, Brasil}
\affil[2]{Laboratoire de Math\'ematiques, INSA Rouen Normandie, 76800 Saint-Etienne du Rouvray, France}
\affil[3]{Programa de P\'os-Gradua\c c\~ao em Matem\'atica, Universidade Federal do Paran\'a, Brazil}

\date{\ }

\begin{document}
\maketitle

\begin{abstract}
A graph is called bivalent or trivalent if there exists an eigenvector of the graph Laplacian composed from $\{-1,1\}$ or $\{-1,0,1\}$, respectively. 
These bivalent and trivalent eigenvectors 
are important for engineering applications, in
particular for vibrating systems. In this article, we 
determine the structure of bivalent and trivalent graphs in the 
following planar graph families: trees, unicyclic, bicyclic, and cactus. 
\end{abstract}

{\bf Keywords:} \\
Laplacian matrix, Laplacian eigenvector, Bivalent graph, Trivalent graph. \\
{\bf MSC classification 2010:} 05C50, 05C35



\section{Introduction}\label{intro}

Let $G=(V,E)$ be a connected graph with vertex set $V=\{1,2,\ldots,n\}.$ 
The Laplacian matrix of $G$ is given by $L(G) = D-A,$ where 
$D=diag(d_{1},\ldots,d_{n})$ is the diagonal matrix with entries equal 
to the degree of each vertex $i \in V$ denoted by $d_{i}$ and $A$ is the 
adjacency matrix. It is well-known that $L(G)$ is symmetric, semidefinite 
positive and the smallest eigenvalue is equal to zero. The 
eigenvalues of $L(G)$ can be arranged as 
$\lambda_1 =0 \leq \lambda_2  \cdots \leq \lambda_{n}.$ 
The eigenspace associated with an eigenvalue $\lambda$ of $L(G)$ 
is denoted by $\mathcal{E}_{L}(\lambda) = \{\mathbf{x} \in 
\mathbb{R}^{n}, \; \mathbf{x}\ne 0: L(G)\mathbf{x} = \lambda \mathbf{x} \}.$ 
The first eigenvalue $\lambda_1=0$ is associated to an 
eigenvector with equal entries; it's multiplicity is the number of connected 
components of the graph \cite{crs01}.

Relating the eigenspace of a graph to its structural properties is an
important topic of spectral graph theory, and there has been a 
growing interest in this subject in recent years. Eigenvectors
of the graph Laplacian are also important in applications
such as vibrations of chemical molecules \cite{crs01}-\cite{cks13}.
In particular, eigenvectors with entries from 
sets $\{-1,1\}$ and $\{-1,0,1\}$ can lead to the concen-\\-tration of
energy when nonlinearities are added to the model \cite{ckka18},\cite{ckk22}.
For all these reasons, we consider the following problem: \\
\noindent \textit{ {\bf P}: Which connected graphs affording a
Laplacian eigenvalue $\lambda$ have eigenvectors
with entries only from sets $\{-1,1\}$ and $\{-1,0,1\}$?} \\

Graphs with Laplacian or signless Laplacian 
eigenvector entries from the sets $\{-1,1\}$ or $\{-1,0,1\}$ have been 
studied in \cite{WHD},  \cite{AL21}, \cite{ALN23}, \cite{barik2011} and
\cite{dam18}. 
The problem above has also been adressed for the adjacency matrix $A$
instead of the graph Laplacian.
Akbari \emph{et al.} in \cite{AAGK06} proved that 
the null space of the adjacency matrix of every forest 
has a $\{-1, 0, 1\}$- basis. Sanders in \cite{S08} proved that for 
every cograph there exist bases of the eigenspaces for the 
eigenvalues $0$ and $-1$ that consist only of vectors with 
entries from $\{-1,0,1\}.$ Sander, Sander, and Kriesell \cite{SSK07} 
characterized all unicyclic graphs for which there exists a basis of 
the kernel that consists only of vectors with entries from $\{-1,0,1\}.$

We recall the following \\
\begin{definition}
A bivalent (resp. trivalent) eigenvector of the graph Laplacian is
composed of $\{-1,+1\}$ (resp. $\{-1,0,+1\}$). \\
A graph affording a bivalent (resp. trivalent) eigenvector is termed 
bivalent (resp. trivalent). 
\end{definition}
In this article, we answered the problem {\bf P} 
by characterizing bivalent and trivalent graphs in some classical families of planar graphs. Planar graphs are of particular interest for many applications. In particular, we characterize bivalent and trivalent trees, unicyclic, bicyclic, and cactus graphs.  Our results indicate that very restricted Laplacian eigenvalues ($\lambda \in \{1,2,3,4\}$) 
are associated with bivalent or trivalent eigenvectors in all studied families.

The article is organized as follows. Section 2 presents preliminary results,
in particular the edge principle by Merris \cite{merris98} which will be useful for
many of our proofs. We also characterize bivalent and trivalent graphs with a leaf.
Section 3 characterizes bivalent and trivalent trees and    
unicyclic graphs. Building on these results, Section 4 identifies all
bivalent and trivalent bicyclic graphs. We consider general
multicyclic graphs in Section 5 and in particular we give a formula for the
cyclomatic number $c$ of regular bipartite graphs, showing which values of
$c$ are possible. We then characterize bivalent and trivalent cactus graphs.

\section{Preliminaries}\label{sec:prel}

We use the terminology vertices and nodes interchangeably as well as 
edges and links. Let $\mathbf{v}=(v_1,\ldots,v_{n})^{T}$ be an 
eigenvector of $L(G),$ we recall the following.
\begin{definition}
A vertex $j \in V$ is a soft node for an eigenvector $v$ of 
the Laplacian of $G$ if there exist a vertex $j$ such that
$v_{j} = 0$. In other cases, the vertex $j$ is called a non-soft node. 
\end{definition}
\begin{definition}[Soft regular graph]
A graph is $d$-soft regular for an eigenvector $\mathbf{v}$ of the Laplacian if every non-soft node for $\mathbf{v}$ has the same degree $d$.
\end{definition}
Vertices $i$ and $j$ that have $v_{i}= v_{j}$ are called equal nodes. An edge connecting equal nodes is called an equal link. The next result is the edge-principle due to Merris \cite{merris98}.

We use the following two important results from Merris
and reformulate the statements.
\begin{theorem}[\textbf{Edge principle} \cite{merris98}]
\label{L2S}
Let $\mathbf{v}$ be an eigenvector of $L(G)$ affording an eigenvalue $\lambda$.
If $v_i=v_j$, then $\mathbf{v}$ is an eigenvector of $L(G^{\prime})$ affording the eigenvalue $\lambda$,
where $G^{\prime}$ is the graph obtained from $G$ by deleting 
or adding the edge $e_{ij}$ depending whether the equal link
$e_{ij}$ is an edge of $G$ or not.
\end{theorem}
\begin{theorem}[Extension of soft nodes, \cite{merris98}]
\label{articulation}
Connecting one or more soft nodes to an existing soft node of a graph
does not change the eigenvalue.
\end{theorem}


From \cite{AL21} and \cite{dam18}, we recall general characterizations of bivalent and trivalent graphs.
\begin{theorem}[\textbf{Bivalent graphs}]
\label{carbi}
The bivalent graphs are the regular bipartite graphs, and their extensions obtained by adding edges between nodes having the same value for a bivalent eigenvector.
\end{theorem}

For trivalent graphs, we need the definition of 
\begin{definition}[Hard degree]
The hard degree $\bar d_j$ of a vertex $j$ is the number of neighbors of $j$
that are not zero. 
\end{definition}

\begin{theorem}[\textbf{Trivalent graphs}] \label{thm:tri_graphs}
Let $G$ be a graph with an eigenvector $\mathbf{v}$ associated with 
Laplacian eigenvalue $\lambda$. Then $G$ is trivalent without 
any equal links if and only if \\
(i) for each vertex $j$ valued 1 or -1, $$\lambda = d_j + \bar d_j,$$
where $d_j, \bar d_j$ are respectively the degree and 
hard degree.\\
(ii) Each soft vertex $j$ is connected to the same number of 
vertices $+1$ and $-1$.
\end{theorem}
The path $P_2$ has a bivalent eigenvector $[-1,1]^{T}$ associated with $\lambda = 2.$ Notice that $\cup_{i=1}^{k} P_2$ still has a bivalent eigenvector $\mathbf{v} = [-1,1,\ldots,-1,1]^{T}$ associated with $\lambda=2$. Connecting equal nodes related to $\mathbf{v}$ generates a graph that has the same eigenvector $\mathbf{v} = [-1,1,\ldots,-1,1]^{T}$ according to Theorem \ref{L2S}.
A similar analysis holds if the initial graph $G$ is a star graph $S_{2k+1}$. In this case, we have $\lambda=1$ associated with 
$\mathbf{v} = [0,-1,1,\ldots,-1,1]^{T},$ where the soft node is 
the dominant vertex. To conclude, using theorem \ref{L2S} we can 
obtain arbitrarily many graphs with eigenvalue
$\lambda=2$ by connecting a number of $P_2$ or graphs with eigenvalue
$\lambda=1$ by connecting star graphs $S_{2k+1}$.

We have the following new result which will be useful in the
following proofs. 
\begin{theorem}[\textbf{Graph with a leaf}]
\label{lgraph}
Let $G$ be a connected bivalent or trivalent graph without equal links.
If $G$ has a leaf, then it is a chain $P_2$ ($[1,-1]^T, \lambda=2$)
or a chain star $S_{2k+1}$ ($ [0,-1,1,\ldots,-1,1]^{T}, \lambda=1)$ with a soft center.
\end{theorem}
\begin{proof}
Consider a leaf $i$, it has value $-1$ or $+1.$ {Assume without loss of generality that $i$ has value 1.} So the adjacent vertex $j$ can only have values
$-1,0$. \\
If $j$ has value -1, then $d_i + \bar d_i = 2= \lambda$. For $j$ we also
have $d_j + \bar d_j = 2$ so that $j$ cannot be connected to another vertex than $i$. \\
If the adjacent vertex $j$ has value 0, then $d_i + \bar d_i = 1= \lambda$ and $j$
needs to have as many neighbors 1 and -1 and they all must have degree 1.
This corresponds to $j$ being the center of a star $S_{2k+1}$.
\end{proof}

\section{Trees and unicyclic graphs} 

In this section, we obtain a characterization of all bivalent and trivalent trees and unicyclic graphs. 

Next result is an immediate consequence of Theorem \ref{lgraph} for trees.

\begin{theorem}[\textbf{Bivalent or trivalent tree}]
\label{trivalenttree}
Bivalent and trivalent trees are obtained respectively 
by connecting chains $P_2$ ($[1,-1]^T, \lambda=2$) or 
star $S_{2k+1}$ ($ [0,-1,1,\ldots,-1,1]^{T}, \lambda=1)$ with equal links. \\ 
More precisely, bivalent trees are obtained by connecting $p \ge 1$ chains $P_2$ with $p-1$ equal links without creating a cycle. 
Trivalent trees are obtained by connecting $p \ge 1$ stars with a soft center with $p-1$ equal links without creating a cycle.
\end{theorem}
The proof is obvious because a tree has at least a leaf. Theorem \ref{trivalenttree} implies that bivalent trees have a perfect matching. 
\begin{figure}[h]
\centering
\includegraphics[width=10cm, height=4cm]{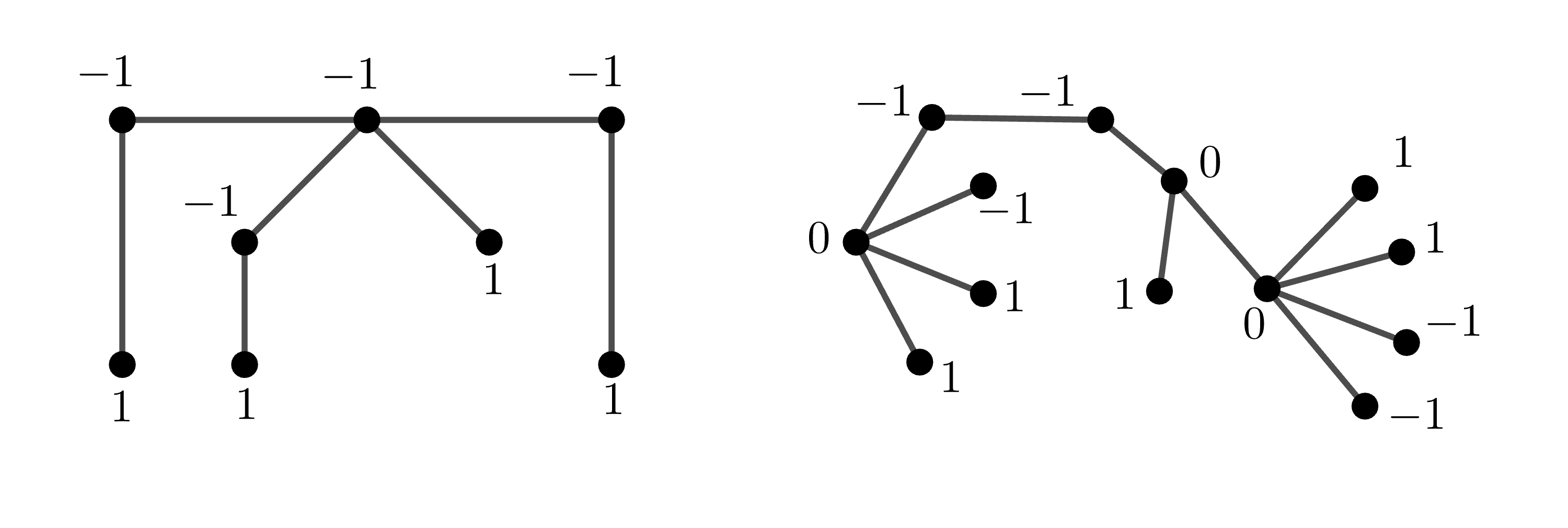}
\caption{Examples of bivalent and trivalent trees with eigenvalues 
$\lambda=2$ (left) and $\lambda=1$ (right).} 
\label{fig:triuni}
\end{figure}

For unicyclic graphs, we then have the following.
\begin{theorem}[\textbf{Bivalent or trivalent unicyclic graphs without equal links}]
\label{trivalentuni}
Bivalent or trivalent unicyclic graphs without equal links are the cycles 
\begin{itemize}
\item Cycle $4k$: eigenvalue 2 associated with    $[1,0,-1,0,1,0,-1,0,\dots,1,0,-1,0]^T;$ 
 
\item Cycle $3k$: eigenvalue 3 associated with $[1,0,-1,\dots,1,0,-1]^T;$
\item Cycle $2k$: eigenvalue 4 associated with $[1,-1,1,-1,\dots,1,-1]^T.$
\end{itemize}
\end{theorem}
\begin{proof}
A consequence of the theorem on a graph with a leaf is that the unicyclic graph cannot
have a leaf because otherwise it would not have a cycle. Then the graph is reduced to a cycle, and $d_{i}=2$ to every vertex $i.$  It is well-known that $\lambda \leq 4$ for cycles.

Let $k$ be the number of nodes valued $1$, which is equal to the number of nodes valued $-1$. The proof is based on Theorem \ref{thm:tri_graphs}. 

We first suppose that $G$ is a trivalent cycle without equal links. Since there is at least one non-soft vertex $v$ of degree 2, from Theorem \ref{thm:tri_graphs}, we have $\lambda = d_v + \overline{d}_{v} \geq 2.$ Then, $2 \leq \lambda \leq 4.$

If $\lambda = 2$, then $\bar{d}_{j}=0$ for every non-soft node $j$. From Theorem \ref{thm:tri_graphs} (ii), the number of nodes valued -1 is also equal to $k.$ It implies that the number of soft nodes equals $2k$ since each node valued $1$ or $-1$ is linked to two soft nodes, and the cycle is $C_{4k}$. 

If $\lambda=3,$ we have $\bar{d}_{j} = 1$ for every non-soft node $j$. It implies that the number of soft nodes is equal to $k$ since each non-soft node is linked to exactly one soft node, and then the cycle is $C_{3k}$.  

If $\lambda=4,$ we have $\bar{d}_{j} = 2$ for every non-soft node $j$, which implies that there are no soft nodes since every vertex has degree 2. So, there is no trivalent graph with this eigenvalue.

Now, suppose that $G$ is a bivalent cycle. From Theorem \ref{carbi}, $G$ is regular bipartite with parts of size $k.$ It implies that $G$ is the cycle $2k,$ and the eigenvalue $\lambda=4$ is associated with eigenvector $[1,-1,\ldots,1,-1]^{T}.$
The proof is complete.
\end{proof}

To conclude, using theorem \ref{L2S} we have the following.
\begin{theorem}[\textbf{Bivalent or trivalent unicyclic graphs}] 
\label{bitrivalent} ~ \\
(i) Bivalent unicyclic graphs are composed of chains $P_2$ connected by equal links that form only one cycle (eigenvalue 2)
or an even cycle (eigenvalue 4). \\ \\
(ii) Trivalent unicyclic graphs are obtained from the following graphs
\begin{itemize}
\item stars $S_{2k+1}$ with 0 in the center and an equal number $k$ of vertices $+1$ and $-1$, with eigenvalue 1;
\item cycle $4k$, with eigenvalue 2; 
\item cycle $3k$, with eigenvalue 3. 
\end{itemize}
and adding equal links {or extending with soft nodes}, not forming a new cycle.
\end{theorem}
\begin{proof}
To prove these two results, consider a bivalent or trivalent unicyclic graph and take out
equal links. There is at most one connected component with a cycle and
all other components are trees. 
All these components are bivalent or trivalent. 
The component with a cycle, as it has no equal links, 
follows the cases of Theorem \ref{trivalentuni}. \\
The bivalent or trivalent trees are described in Theorem \ref{trivalenttree}.
\end{proof}

Figure \ref{fig:triuni} shows a unicyclic trivalent graph obtained by 
extending with soft nodes in red color. 
\begin{figure}[h]
\centering
\includegraphics[width=6cm, height=4cm]{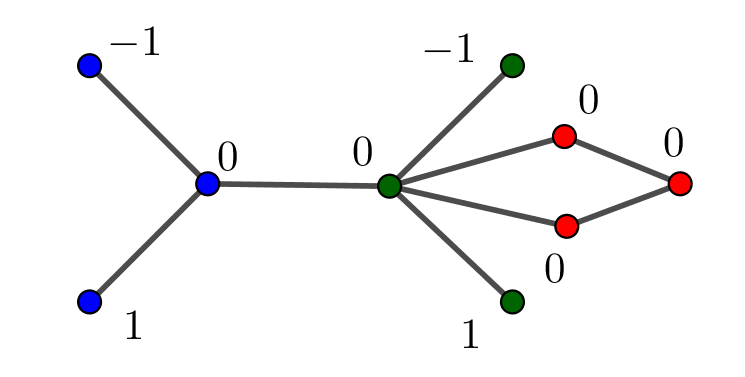}
\caption{A trivalent unicyclic graph with eigenvalue 1.} 
\label{fig:triuni}
\end{figure}
\begin{figure}[h]
\centering
\includegraphics[width=7cm, height=3cm]{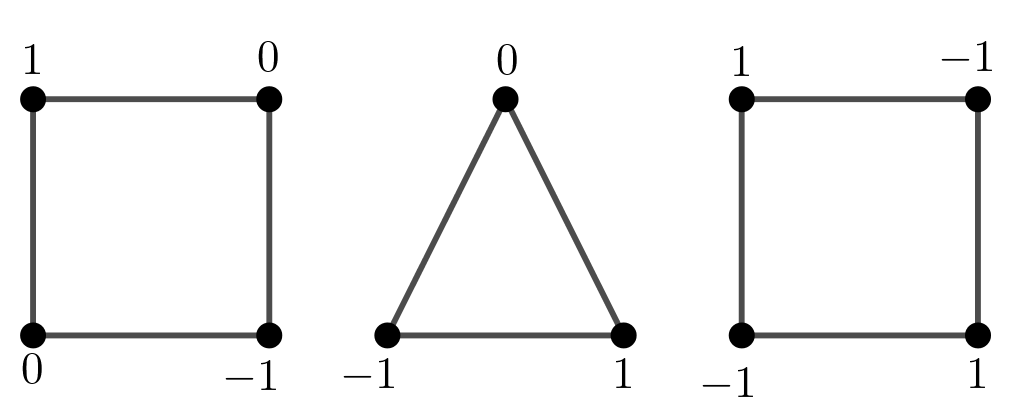}
\caption{The trivalent (first two) and bivalent elementary cycles, with eigenvalues 2,3 and 4 from left to right.} 
\label{fig:l234}
\end{figure}

\section{Bicyclic graphs}

A connected graph $G$ of order $n$ is called a bicyclic graph if $G$ 
has $n+1$ edges. Because of Theorem \ref{lgraph} bivalent or trivalent bicyclic
graphs cannot have a leaf.

It is known \cite{He2007} that there are three types of bicyclic graphs 
containing no leaves. These are shown in Figure \ref{fig1} and
noted $B_1, B_2$ and $B_3$.
Let \textcolor{black}{$B_{1}(p,q)$} be the set of bicyclic graphs 
obtained from two vertex-disjoint cycles $C_p$ and $C_q$ by identifying 
a vertex $u$ of $C_p$ and a vertex $v$ of $C_q$ such that 
$n = p+q-1.$ \\
\textcolor{black}{Let $B_{2}(p,q,r)$ be the bicyclic graph obtained from two 
vertex-disjoint cycles $C_p$ and $C_q$, by joining vertices $v_1$ of 
$C_p$ and $u_r$ of $C_q$ by a path of length $r-1$ such that  
$n = p+q+r$.} \\
Let $B_{3}(p,q,r)$ be the bicyclic graph obtained by taking three chains 
containing respectively $p, q$ and $r$ vertices  and identifying
their two extremities.\\
Notice that $B_1(p,q)$ is a special case of $B_2(p,q,r)$ for $r=0$.
\begin{figure}[!h]
    \centering
\includegraphics[height=8cm]{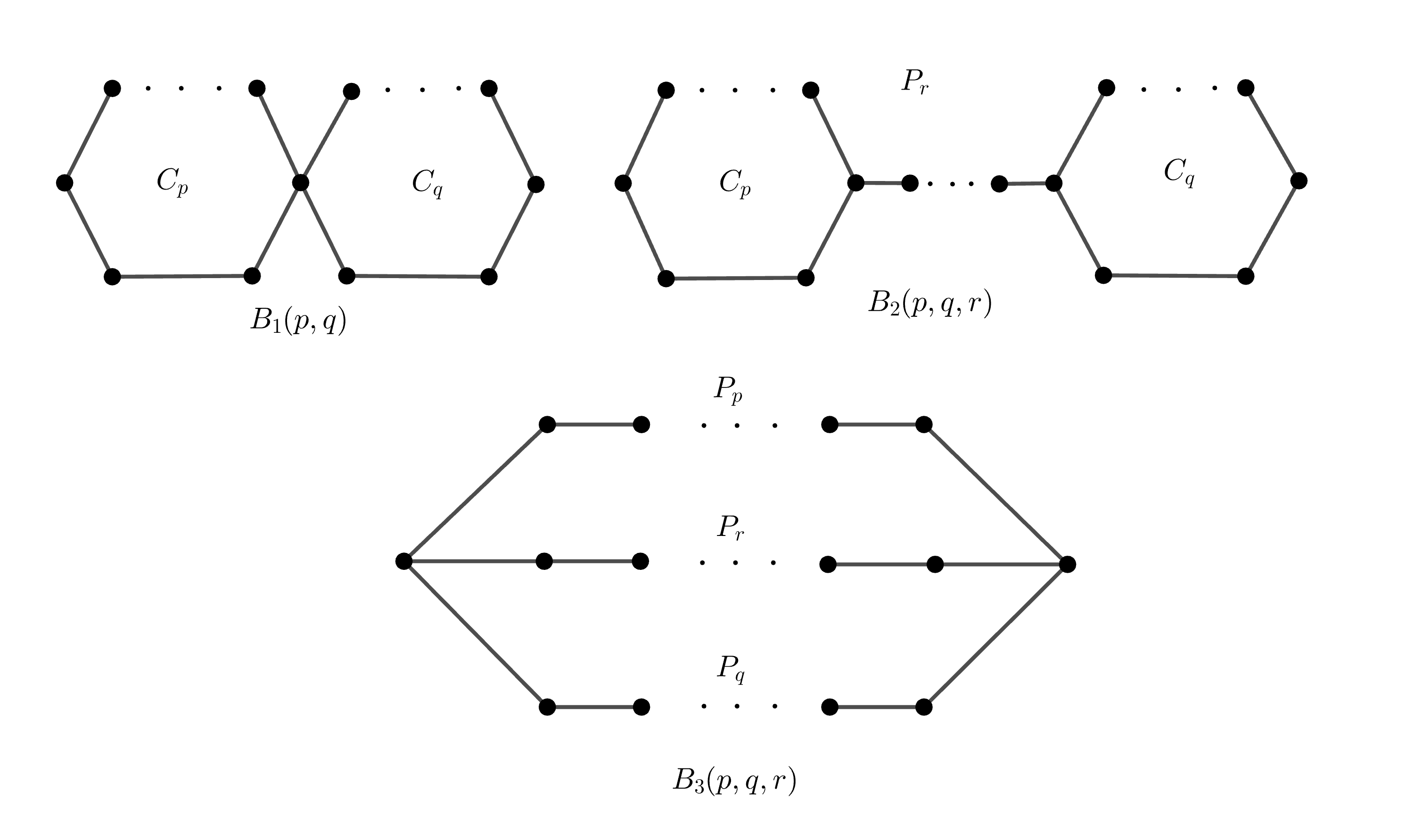}
    \caption{Families $B_1(p,q)$, $B_2(p,q,r)$ and $B_3(p,q,r)$ of bicyclic graphs with no pendant vertices}
    \label{fig1}
\end{figure}

\begin{lema}\label{bivalentbicyclicnolinks}[\textbf{Bivalent 
bicyclic graphs without equal links}] If $G$ is a bicyclic graph without 
equal links, it cannot be bivalent.
\end{lema}
\begin{proof}
Let $G$ be a bicyclic graph. From Theorem \ref{lgraph}, $G$ has no 
leaf, so it should belong to $B_1(p,q), B_2(p,q,r)$ or $B_{3}(p,q,r).$ Notice that in every case it is not possible to obtain a vertex bipartition of $V(G)$ such that $G$ is a regular bipartite graph since there is at least one vertex $u \in V(G)$ such that $d_{u} = 3$ while there are some other vertices with degree 2. From Theorem \ref{carbi}, $G$ is not bivalent, and the proof is complete. \end{proof}


\textcolor{black}{\begin{theorem}[\textbf{Bivalent bicyclic graphs}]
\label{bivalentbicyclic}
Bivalent bicyclic graphs are
composed of either\\
(i) chains $P_2$ connected by equal links that form only two cycles
(eigenvalue $\lambda=2$) \\
(ii) one even cycle with an additional equal link 
(eigenvalue $\lambda=4$). \\
(iii) two even cycles connected by an equal link $\lambda=4$).
\end{theorem}}
\begin{proof}
Let $G$ be a bicyclic graph. Taking out equal links, we obtain $p$ connected
components. 
Each component is either a tree, a unicyclic graph (recall that a bicyclic
graph has necessarily an equal link) or a bicyclic graph 
(in this case $p=1$).   \\
If the component is a tree, by Theorem \ref{trivalenttree} it is necessarily a chain $P_2$
and $\lambda=2$. \\
If the component is a unicyclic graph, from Theorem \ref{trivalentuni}, it is an even cycle
and $\lambda=4$. 
\end{proof}

\begin{figure}[h]
\centering
\includegraphics[height=4cm, width=10cm]{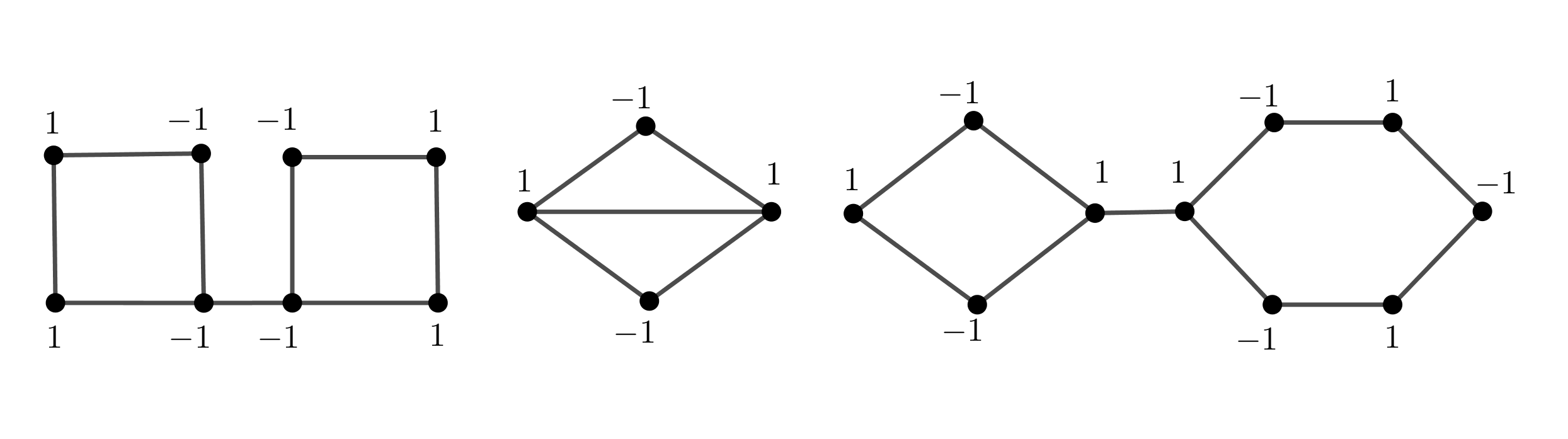}
\caption{Bivalent bicyclic graphs from left to right 
(i) chains $P_2$ connected by equal links, $\lambda=2$, 
(ii) even cycle with an additional equal link $\lambda=4$ and
(iii) two even cycles connected by an equal link $\lambda=4$.
}
\label{fig:bicyclic-bi}
\end{figure}

Next, we characterize trivalent bicyclic graphs without equal link.

\begin{theorem}[\textbf{Trivalent bicyclic graphs without equal links}]
\label{trivalentbicyclic}
Trivalent bicyclic graphs without equal links are:
\begin{itemize}
\item[(i) ] $B_{1}(p,q)$: \\
eigenvalue $\lambda=2$ with 
$(p,q) = (4k_1, 4k_2)$ or  $(4k_1+2, 4k_2+2)$, for any integers $k_1, k_2 \geq 1.$ \\
eigenvalue $\lambda=3$ with $p=3k_1$ and $q=3k_2$ for any integers $k_1, k_2 \geq 1;$ 
\item[(ii) ] $B_{3}(p,q,r):$ \\
eigenvalue $\lambda=3$ with 
$(p,q,r) = (3k_1, 3k_2, 3k_3)$  for any integers $k_1, k_2, k_3 \geq 1.$ \\
eigenvalue $\lambda=4$ with \textcolor{black}{$(p,q,r) = (3, 2, 3)$} (diamond graph)
\end{itemize}
\end{theorem}
\begin{proof} 
From Theorem \ref{lgraph}, a bicyclic trivalent graph cannot have a leaf, then $G$ should 
belong to one of the following families: $B_1(p,q), B_{2}(p,q,r)$ or 
$B_{3}(p,q,r)$ for $p,q,r \geq 3$.

Consider a cycle of these graphs, it has at least three vertices and
two neighbors cannot have the same value so there is at least 
one vertex $v$ of degree 2. If at least one of the cycles is of length
4 or more, then there are at least two vertices of degree 2 that 
are neighbors and one of them is non-zero. 
Then we have $d_v + \bar d_v = \lambda \leq 4$. \\
In the particular case where both cycles are of length 3, we have the
diamond graph with $\lambda=4$, see Fig. \ref{fig:bicyclic_trivalent}. \\
For cycles of length greater than 3, we have for any non-zero vertex $v$ of degree 2 in a cycle
$$2 \leq d_v + \bar d_v = \lambda \leq 3 .$$
To show this, assume $x_v=1$ and $d_v + \bar d_v =4$, then the neighbors of $v$ all have
$x=-1$ and by induction, all nodes of the cycle have value 1 or -1, including the joint node $u$.
Node $u$ has three neighbors (cases $B_2$ or $B_3$) or four neighbors 
(case $B_1$) neighbors so that $d_u + \bar d_u > 4$.
This is impossible.

Consider a graph of the kind $B_1$ or $B_2$.\\
Assume $x_u \neq 0$. 
Then for graphs of the kind $B_1$ $d_u=4$ so 
$d_u + \bar d_u \ge 5$ which is impossible. \\
For graphs of kind $B_2$, $d_u + \bar d_u \ge 5$ and again it is impossible.
Then $x_u=0$. \\
For graph $B_2$ with $l>0$, consider the node $u$ of degree three, it's 
neighbor 
on the chain has to be of value 0 so that the sum of the values of the
neighbors of $u$ is zero. 
This is impossible because we assumed no equal link.
Therefore $B_2$ with $l>0$ is ruled out. \\
Assume $\lambda=2$. \\
Then for any node $v,~ x_v\neq 0$ in a cycle, $\bar d_v =0$ so its neighbors $w$ verify
$x_w=0$.
The cycle is then formed by repetitions of the pattern
$$ (1 , 0, -1, 0) . $$
Now assume $\lambda=3$, then for any node $v,~ x_v=1$ of a cycle, 
one of its neighbors has value 0 and the other verifies $x=-1$. Therefore
cycles are made of patterns
$$ (0, 1,-1) . $$

For graphs $B_3$, node v of degree 3 cannot have value 0 because then it 
should have as many neighbors of value 1 as neighbors of value -1.
Assume $x_v=1$. 
Now assume $d_v + \bar d_v = \lambda = 3 + \bar d_v$, the only possibility is
$\bar d_v=0$ and $v$ has only neighbors with value 0. Therefore $\lambda = 3 $.
Since the graph is trivalent, on one of the chains, vertices with value 0
can only have neighbors 1 and -1. To conclude, each chain is formed by
patterns
$$(1,0,-1). $$
including the two vertices $u$ and $v$ of degree 3. Therefore $x_v=1=-x_u$.\\
\end{proof}
Figure \ref{fig:bicyclic_trivalent} shows examples of trivalent 
bicyclic graphs without equal link.
\begin{figure}[!ht]
\centering
\includegraphics[height=10cm,angle=0]{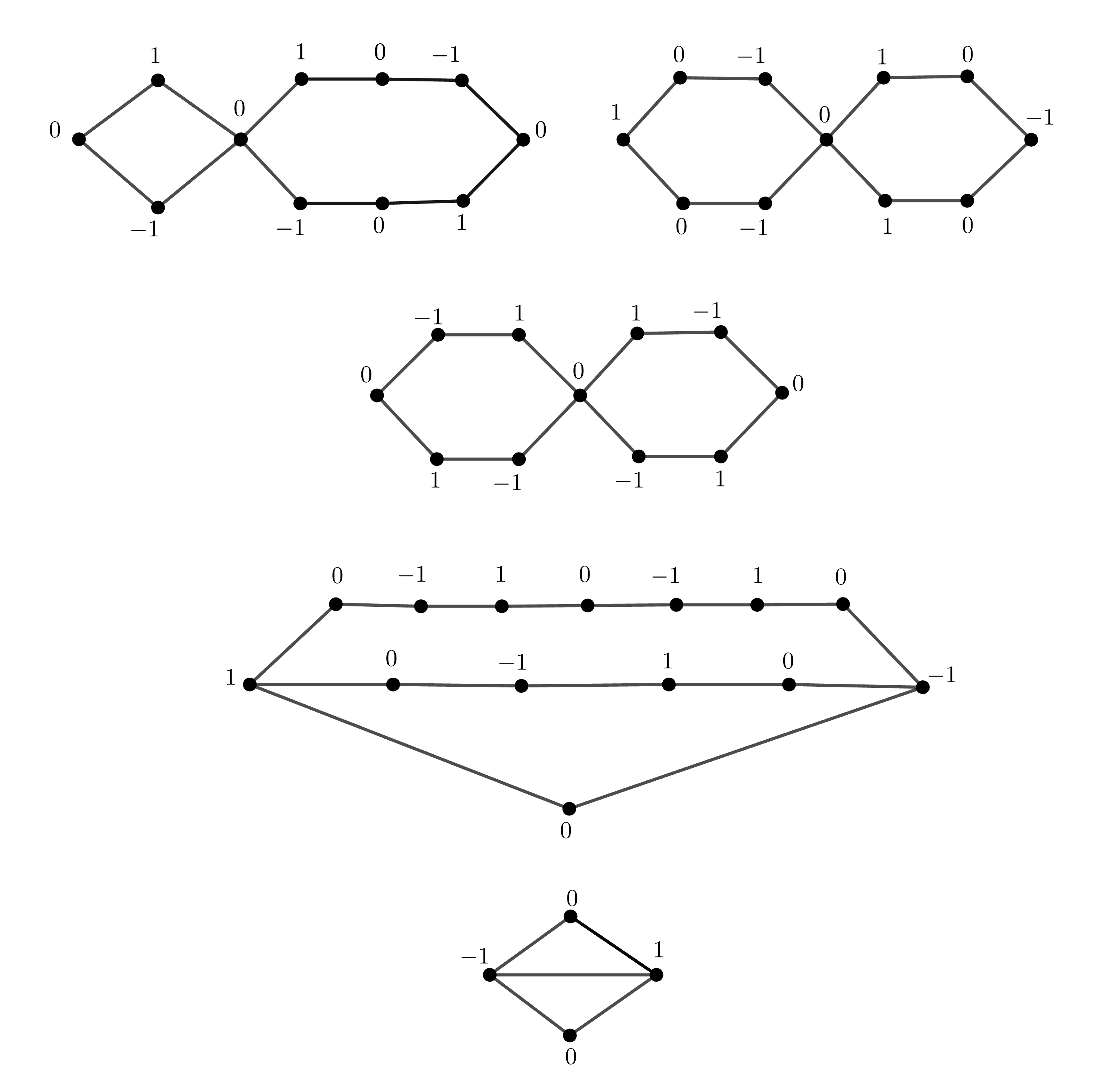}
\caption{ Bicyclic trivalent graphs without equal link:
1st row, $B_{1}(4,8)$ and $B_{1}(6,6)$ for $\lambda=2$ (case (i));
2nd row $B_{1}(6,6)$ for $\lambda=3$ (case (i)); 3rd row 
$B_{3}(9,6,3)$ for $\lambda=3$ (case (ii)) and 4th row 
diamond graph $B_{3}(3,2,3)$ for for $\lambda=4$ (case (ii)).
}
\label{fig:bicyclic_trivalent}
\end{figure}
\newpage
Note that we can add an arbitrary number of chains to $B_3$ connected
to the same vertices u and v and form multicyclic trivalent graphs without 
equal link.

Now, we are ready to state the characterization of trivalent 
bicyclic graphs, which is a consequence of 
Theorems \ref{bivalentbicyclic} and \ref{trivalentbicyclic}.

\begin{theorem}[\textbf{Trivalent bicyclic graphs}]
\label{trivalentbicyclicgeneral}
Let $k_1, k_2, k_3 \geq 1$ be integers.  Trivalent bicyclic graphs are obtained from the following graphs
\begin{itemize}
\item [(i)] stars $S_{2k+1}$ with 0 in the center and an equal number $k$ of vertices $+1$ and $-1$, with eigenvalue 1;

\item [(ii)] $B_{1}(3k_1, 3k_2)$ with eigenvalue 2;

\item [(iii)] $B_{1}(4k_1, 4k_2)$ or $B_{1}(4k_1+2, 4k_2+2)$ with eigenvalue $3;$

\item [(iv)] $B_{3}(3k_1, 3k_2, 3k_3)$ with eigenvalue $3,$
\end{itemize}
{and adding links between equal links or extending with soft nodes, not forming a new cycle.}
\end{theorem}

\begin{figure}[h]
\centering
\includegraphics[width=14cm, height=7cm]{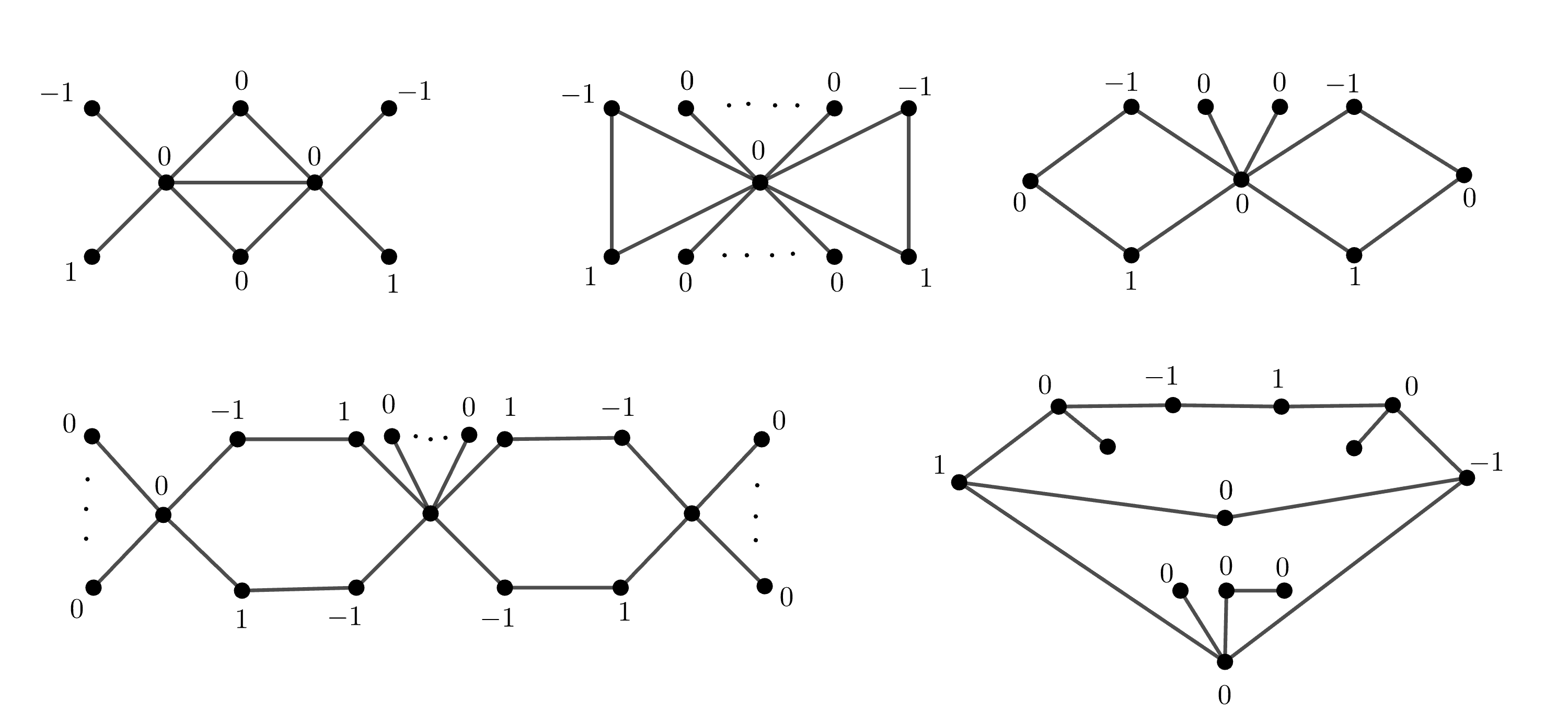}
\caption{Examples of trivalent bicyclic graphs, eigenvalues are
1,3,3,3 and 3 from top left to bottom right.} 
\label{fig:triuni2}
\end{figure}

\section{Multicyclic graphs}

\subsection{General results}

We show below that a bivalent graph without equal links
can be decomposed in perfect matchings.
\begin{theorem}[\textbf{Partition in perfect matchings}]

\label{perfMatch}
Let $G$ be a $d$-regular bipartite graph, then the edges
can be partitioned in $d$ perfect matchings.
\end{theorem}
\begin{proof}
The theorem is trivially true for $d=0$ (no edge)
or $d=1$ (matching). For $d>1$, we use the fact that
a regular bipartite graph has a perfect matching. This is
a consequence of K\"onig's matching theorem (see for example
Schjriver's book \cite{sch03}, corollary 16.2.b).
Taking out a perfect matching gives a $d-1$ regular bipartite
graph, hence the conclusion.
\end{proof}

We note that two disjoint perfect matchings in $G$ form a cycle cover of the graph (not necessarily
just one cycle). 
If the two perfect matchings complete each other well, they may form
a Hamiltonian cycle. \textcolor{black}{
Note that a regular bipartite graph is not necessarily Hamiltonian, as shown
in the counterexample in Figure \ref{fig:counterexample} built from 
three components $K_{3,3} \setminus P_2$ and two additional
vertices $\alpha,\beta$ linked to the vertices of degree 2.
This graph is regular and bipartite:  the vertex set is
$$V=(A \cup \{\alpha\}) \cup (B \cup \{\beta\})$$ .}
\begin{figure}[h]
\centering
\includegraphics[width=6cm, height=7cm]{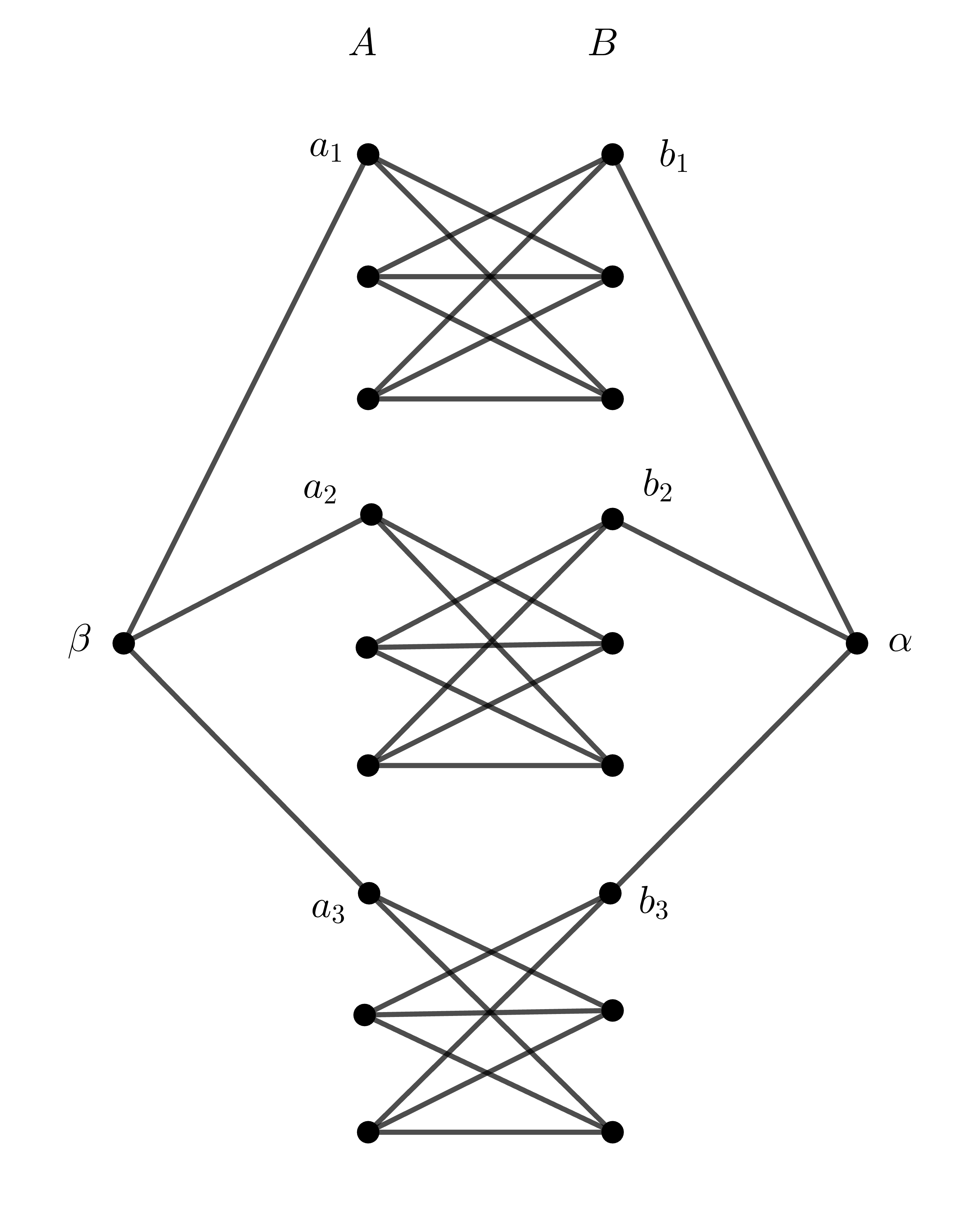}
\caption{A 3-regular non-Hamiltonian graph with two disjoint perfect matchings} 
\label{fig:counterexample}
\end{figure}

The cyclomatic number $c$ of a connected graph is defined as the number of independent cycles, and it can also be viewed as the minimum number of edges to be removed to get a tree. So,
$m = n-1+c.$

We now give a formula for the cyclomatic number $c$ of regular bipartite graphs. This
will help finding the possible values of $c$.
\begin{theorem}[\textbf{Regular bipartite multicyclic graphs}]
\label{regbiparmcycle}
The cyclomatic number of regular bipartite graphs verifies 
\begin{equation} \label{cnd}
c = { n  \over 2} (d-2)+1 ,\end{equation}
where $d$ is the degree.
\end{theorem}
\begin{proof}
Assume $G$ to be a connected regular bipartite multicyclic graph.
Then it's number of edges $m$ verifies
$$m = n-1+c,$$
where $c$ is the cyclomatic number. We know that $n$ is even so
if $d$ is the degree of each vertex, we have
$$n d = 2 m .$$
From the two relations above, we get relation (\ref{cnd}).
\end{proof}

Given this formula, one sees that not all values of $c$ are
possible.
\begin{corollary}
The cyclomatic number $c$ of regular bipartite graphs verifies
$$c-1 \in \mathbf{Z}^+ - \{k^2, k \geq 2\}\cup \{ k(k+1), k \geq 1\}.$$
\end{corollary}

The relation above imposes a constraint between the values of
$n,d$ and $c$. Let us consider the first few values of $d$. \\
For $d=0$, we have $n=1$ and $c=0$ so we have isolated vertices.\\
For $d=1$, we have $n=2$ and $c=0$ so the graph is $P_2$. \\
For $d=2$, the graph is a cycle and $c=1$. \\
Note also that $c=2$ is not possible because
$n/2$ and $d-2$ should be different non consecutive integers.
This is theorem \ref{bivalentbicyclic}.\\
Similarly, $c=3$ is not possible.
An example with $c=4$ is $K_{3,3}$ the complete bipartite graph 
where $n=6$ and $d=3$.

An infinite family of regular bipartite graphs is obtained
by taking an even cycle $C_{2k}$ and adding to it
$\ell$ alternate matchings, 
$k,\ell \in \mathbf{N}^*  , ~k \ge 2, ~~\ell \leq n/2 -2$.
These graphs are Hamiltonian. \\
An example is $d=4$, then $c=n+1$ and the graph is a
cycle $C_{2k}$ with $\ell=2$ matchings so $\lambda=8$.

In the following we consider a subclass of multicyclic planar graphs
called cactus graphs. 
A \textit{cactus graph} is
a connected graph in which any two cycles have at most one vertex in common. Such a graph is outerplanar: it can be represented with all vertices belonging
to the outer face. 

\subsection{Cactus bivalent graphs}
First note 
\begin{theorem}\label{bivalentcactusnolinks}[\textbf{Bivalent 
cactus graphs without equal links}] If $G$ is a cactus graph with more
than one cycle and without equal links, it cannot be bivalent.
\end{theorem}
\begin{proof}
We show that it is not possible to build a finite d-regular cactus. \\
Consider a bivalent cactus graph without equal links, it is
$d$ regular. We choose an outer planar representation of the graph.
Take a face of this representation, it is a cycle without a chord.
Because the graph is connected and has at least two cycles, the
degree of each vertex is at least 3.
Take a vertex $u$ of this cycle, it is connected to another vertex $v$
outside the cycle. Vertex $v$ has degree at least 3 and is connected
to vertices $v_1$ and $v_2$ not belonging to the first cycle.
Otherwise the cactus hypothesis would be wrong.
Repeating the process for $v_1$ we generate an infinite graph.
\end{proof}

Then we have
\begin{theorem}\label{bivalentcactus}[\textbf{Bivalent 
cactus graphs}] A bivalent cactus graph, is either 
formed by cycles of even length ($\lambda=4$) connected by
equal links or formed by chains $P_2$ connected by
equal links ($\lambda=2$).
\end{theorem}
\begin{proof}
Consider a cactus graph. Taking out equal links
we obtain $p$ connected components. \\
Consider one of these connected components, it cannot have more
than one cycle because of Theorem \ref{bivalentcactusnolinks}. Assume there is just one
cycle, then from Theorem \ref{trivalentuni} this component is just 
a cycle $2k$ ($\lambda=4$). If this component has no cycle then
it is just a chain $P_2$ from Theorem \ref{trivalenttree}.
\end{proof}

\subsection{Cactus trivalent graphs}

As a consequence of Theorem \ref{trivalentbicyclic}, we obtain all trivalent cactus graphs without equal links.
\begin{theorem}\label{trivalentcactusnolinks} [\textbf{Trivalent cactus 
graphs without equal links}] Let $G$ be a trivalent cactus graphs without 
equal links, only three cases are possible:
\begin{itemize}
\item $\lambda=1,$ then $G$ is a chain star $S_{2k+1};$
\item $\lambda = 2$, then $G$ should be obtained by identifying soft nodes of cycles with length multiple of $4;$
\item $\lambda = 3$, then $G$ should be obtained by identifying soft nodes of cycles with length multiple of $3.$

\end{itemize}
\end{theorem}
\begin{proof}
From Theorem \ref{lgraph}, if $G$ has a leaf then it should be a chain star $S_{2k+1}.$ Now, consider that $G$ does not have a leaf. By hypothesis, $G$ cannot have a subgraph of the type $B_3(p,q,r)$; otherwise, $G$ would have two cycles with \textcolor{black}{$r-1$} vertices in common. From Theorem \ref{trivalentbicyclic}, $G$ cannot have a path connecting two cycles since there is no trivalent bicyclic graph in \textcolor{black}{$B_2(p,q,r).$} So, $G$ is composed of cycles identified by, at most, one vertex.  Also, in light of the results of Theorem \ref{trivalentbicyclic}, 
$G$ is obtained by identifying cycles of lengths multiple to $3$ or $4$. Notice that each pair of cycles should be identified by their soft nodes; otherwise, the degree of a vertex $i$ of value $+1$ (or $-1$) would increase while $\lambda$ is fixed, and then $\lambda < d_{i} + \bar{d}_{i}.$ The proof is complete. 
\end{proof}

Figure \ref{fig:trivalentcactus} shows a trivalent cactus graph.

\begin{figure}[!ht]
\centering
\includegraphics[height=5.5cm]{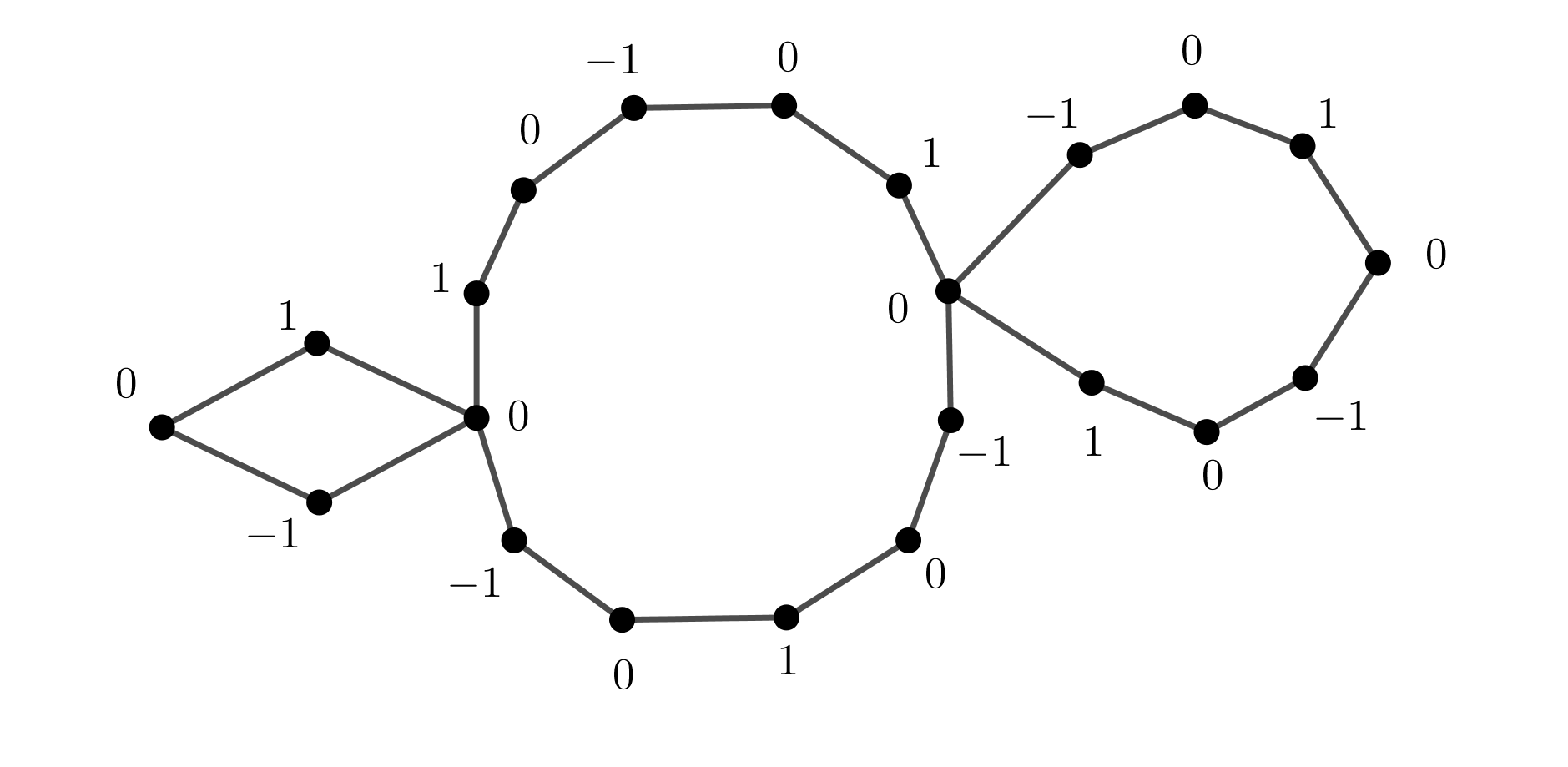}
\caption{Example of a trivalent cactus with eigenvalue 2. }
\label{fig:trivalentcactus}
\end{figure}

Then we have
\begin{theorem}\label{trivalentcactus}[\textbf{Trivalent 
cactus graphs}] Trivalent cactus graph are
\begin{itemize}
\item if $\lambda = 1$, collection of stars $S_{2k+1}$ for $k\geq 1$ connected by equal links;
\item if $\lambda = 2$, $P_2$ and cycles $C_{4k}$ intersecting at zero vertices
and zero vertices, all connected by equal links;
\item if $\lambda = 3$, cycles $C_{3k}$ intersecting at zero vertices
and zero vertices, all connected by equal links,
\end{itemize}
and extending with soft nodes.
\end{theorem}
\begin{proof}
Consider a trivalent cactus graph. Taking out equal links
we obtain $p$ connected components. 
Consider one of these connected components, it falls into
the following three cases: \\
{\bf case 1} Isolated vertices. These vertices all have value 0,
otherwise the conditions for a trivalent graph are not satisfied,
see Theorem \ref{thm:tri_graphs}. \\
{\bf case 2} Bivalent. If no cycle is present, the component is
$P_2$ (see Theorem \ref{trivalenttree}). \\
If a cycle is present, it is unique, and the component is reduced to
this cycle. From Theorem \ref{bivalentcactusnolinks}, this cycle is
$C_{2k}$ and $\lambda=4$. \\
{\bf case 3} Trivalent. Then from Theorems \ref{bitrivalent} and \ref{trivalentcactusnolinks}
the component is either \\
(a) a star $S_{2k+1}$ with $0$ in the center and an equal number $k$ of vertices $+1$ and $-1$, $\lambda=1$;\\
(b) a collection of cycles $C_{4k}$ intersecting at zero vertices, $\lambda=2;$  \\
(c) a collection of cycles $C_{3k}$ intersecting at  zero vertices, $\lambda=3$ . 

\vspace{0.1cm}

\noindent Putting back the links, if there exists a cycle $C_{2k}$ and $\lambda=4$
(case 2) then all connected components should be even cycles
so the graph is bivalent and not trivalent. This contradicts the hypothesis.
This rules out bivalent components consisting of a cycle $C_{2k}$; 
however bivalent $P_2$ components are allowed for $\lambda=2$. \\
To conclude, there should be at least one trivalent component (case 3) connected to 
other components with the same $\lambda$. 
\end{proof}

Note that the zeros forming one of the connected components (case 1) can result in a cycle. Figure \ref{fig:triv_cactus_general} shows a trivalent cactus graph with equal links. The trivalent cactus graph on the left of the first row ($\lambda=2$) is obtained by identifying $C_8$ with $C_4$ in a soft node, and then adding an equal link of the resultant graph to $P_2$, and extending the graph with a cycle $C_5$ where all nodes are soft. The other graph of the first row ($\lambda=3$) is obtained by adding an equal link between $C_6$ and $C_3,$ and extending with soft nodes of a tree and a cycle $C_4.$  The graph of the bottom ($\lambda=1$) is obtained by connecting the star $S_3, S_3$ and $S_5$ by equal links.

\begin{figure}[!ht]
\centering
\includegraphics[height=9cm]{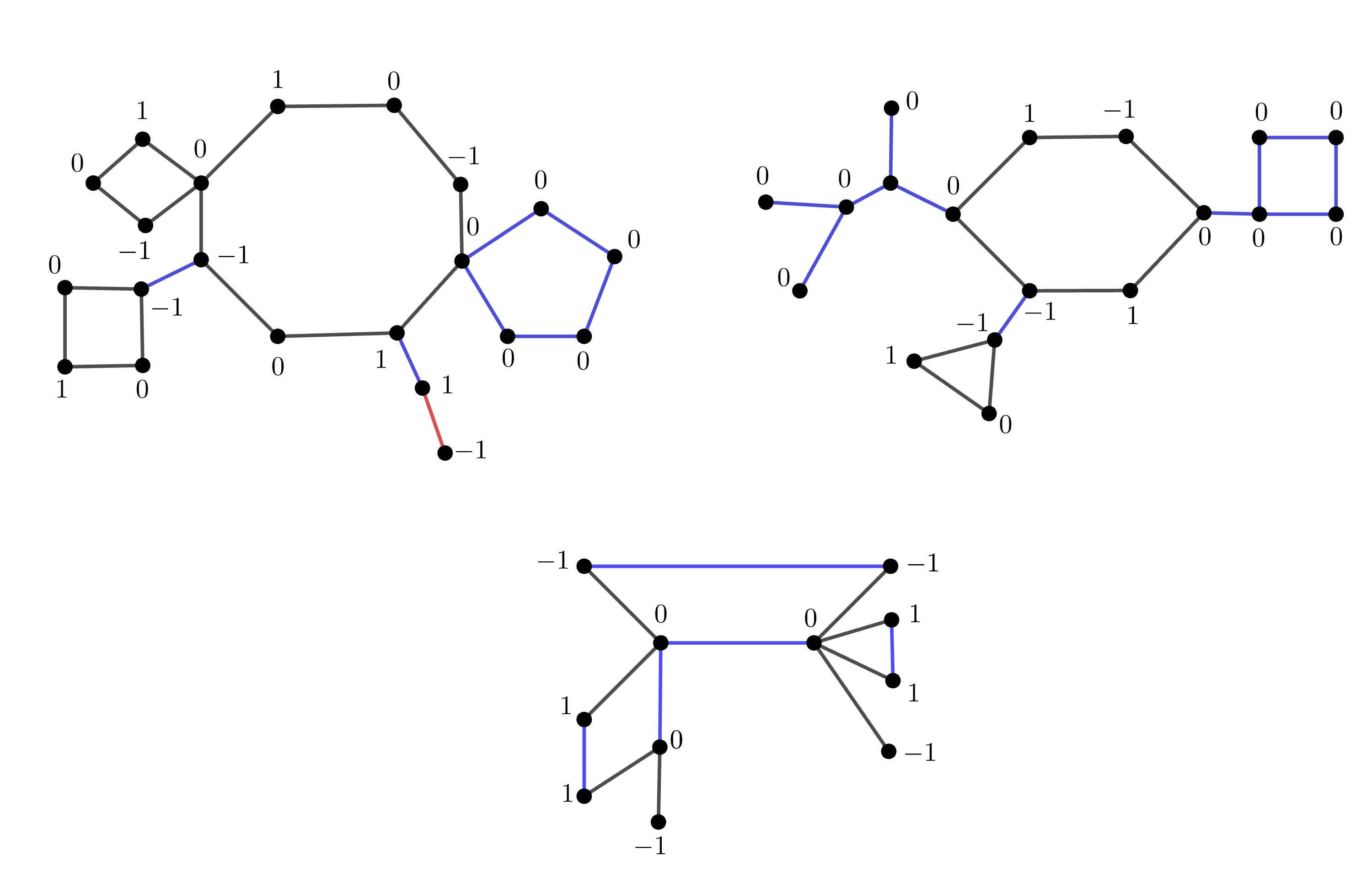}
\caption{Examples of cactus graphs with equal links. Blue links are equal links, and the red link illustrates a $P_2$ connected by an equal link to the cycle $C_8.$}
\label{fig:triv_cactus_general}
\end{figure}

\vspace{1cm}

\noindent{\bf Acknowledgments.} The research of Leonardo de Lima is supported by CNPq grants 315739/2021-5 and 403963/2021-4.\\

\end{document}